\documentclass[12pt, fleqn, a4paper, oneside]{imsart}

\usepackage{amsthm,amsmath,amssymb}
\usepackage{graphicx}

\usepackage[authoryear]{natbib}

\usepackage{algorithm}
\usepackage{algorithmic}

\bibpunct{(}{)}{;}{a}{,}{,}
\RequirePackage[OT1]{fontenc}

\startlocaldefs

\newtheorem{theorem}{Theorem}

\newtheorem{lemma}{Lemma}
\newtheorem{proposition}{Proposition}

\theoremstyle{remark}
\newtheorem{remark}{Remark}

\theoremstyle{definition}
\newtheorem{definition}{Definition}

\endlocaldefs

\begin{document}

\begin{frontmatter}


\title{Unsupervised nonparametric detection of unknown objects in noisy images based on percolation theory}
\runtitle{Unsupervised detection and percolation}

\begin{aug}
\author{\snm{Mikhail} \fnms{Langovoy}\corref{}\thanksref{t2}
\ead[label=e1]{mikhail.langovoy@epfl.ch}}

\affiliation{EPFL, Switzerland}

\address{Machine Learning and Optimization Laboratory\\
       EPFL, Station 14 \\
       Lausanne, CH-1015 Switzerland\\
\printead{e1}}

\and

\author{\snm{Olaf} \fnms{Wittich} \ead[label=e2]{olaf.wittich@matha.rwth-aachen.de}}

\affiliation{RWTH Aachen, Germany}

\address{Lehrstuhl A f\"{u}r Mathematik \\ RWTH Aachen, 52056 Aachen \\
\printead{e2}}

\and

\author{\snm{Laurie} \fnms{Davies} \ead[label=e3]{laurie.davies@uni-due.de}}

\affiliation{University of California at Davis, USA}

\address{Department of Statistics, \\
University of California at Davis, Davis CA, \\
95616-8572, USA\\
\printead{e3}}

\thankstext{t2}{Corresponding author.}

\runauthor{M. Langovoy et al.}
\end{aug}

\smallskip


\begin{abstract}
We develop an unsupervised, nonparametric, and scalable statistical learning method for detection of unknown objects in noisy images. The method uses results from percolation theory and random graph theory. We present an algorithm that allows to detect objects of unknown shapes and sizes in the presence of nonparametric noise of unknown level. The noise density is assumed to be unknown and can be very irregular. The algorithm has linear complexity and exponential accuracy and is appropriate for real-time systems. We prove strong consistency and scalability of our method in this setup with minimal assumptions.
\end{abstract}



\end{frontmatter}

\section{Introduction}\label{Section1}


Detection of objects in noisy images is the most basic problem of image analysis. Indeed, when one looks at a noisy image, the first question to ask is whether there is any object hidden behind the noise, at all. This is also a primary question of interest in such diverse fields as, for example, cancer detection \citep{Cancer_Detection_1}, automated urban analysis \citep{Road_Detection_IEEE}, detection of cracks in buried pipes \citep{Sinha200658}, and other possible applications in astronomy, electron microscopy and neurology. Moreover, if there is just a random noise in the picture, it doesn't make sense to run computationally intensive procedures for image reconstruction for this particular picture. This is especially relevant in modern day applications to Internet data and in automated image processing systems, where one has to mine billions of images under time constraints. Surprisingly, the vast majority of image analysis methods, both in statistics and in engineering, avoid the pure detection problem and start immediately with the more challenging and computationally intensive task of image reconstruction.


\subsection{Related work}

As pixels in digital images can be viewed as network nodes with attributes, many applications from image processing, such as road tracking \citep{Geman_1996_Roads} or medical tumor detection \citep{Mcinerney_1996_Medical_Image}, can be treated within the framework of community detection in  networks. More specifically, these setups correspond to detection of communities hidden in large networks with noisy node attributes, where one decides on the existence of communities using both the network topology as well as the network's content represented by node attributes (see, e.g., \citep{Ruan_2013_Content_Links} or \citep{Yang_2013_Community_Node} for related types of setups).

We are only concerned with the existence of an object in the image, and not with estimating the object. We are also heavily using the fact that images are processed on computers only in a discretized form. For this problem and the setup, recently a new line of research emerged where discrete probability methods of statistical physics were applied to unsupervised community detection in discrete structures such as pixelized images or lattices \citep{langovoy_report_2009-035}, \citep{Langovoy_Wittich_Square}, \citep{langovoy_wittich_robust}, \citep{Arias-Castro_Grimmett}. The idea of applying percolation theory to study of hidden communities in networks, combined with variations of $k$-NN scans, proved useful in application areas like anomaly detection and automated detection of unknown objects in extremely noisy images \citep{langovoy_habeck_schoelkopf_JSM}, \citep{langovoy_habeck_schoelkopf}.




However, in \citep{Arias-Castro_Grimmett} only the case of parametric noise from a one-parameter exponential family was considered, while the present paper deals with the general case of nonparametric noise. Our bulk condition on the object interior is also more general than conditions on cluster sizes in \citep{Arias-Castro_Grimmett}. The algorithm in the present paper has linear computational complexity, irrespectively of a shape of an object. Papers \citep{langovoy_report_2009-035} and \citep{Langovoy_Wittich_Square} treated different kind of underlying lattices and had to resort to the case of nonparametric noise of bounded level, while \citep{langovoy_wittich_robust} had to limit the class of possible noise distributions via additional smoothness assumptions. In this paper, we establish a strong form of consistent detection for a much wider class of noise distributions, and \emph{completely remove smoothness assumptions} on the noise.



\subsection{Main contributions}

From the statistical point of view, we treat the object detection problem as a nonparametric hypothesis testing problem within the class of statistical inverse problems on networks. We assume that the noise density is completely unknown, and that it is not necessarily smooth or continuous. In this paper, we propose an algorithmic solution for this nonparametric hypothesis testing problem. We prove that our algorithm has linear complexity in terms of the number of pixels on the screen, and this procedure is not only asymptotically consistent, but on top of that has accuracy that grows exponentially with the "number of pixels" in the object of detection. The algorithm has a built-in data-driven stopping rule, so there is no need in human assistance to stop the algorithm at an appropriate step.

The crucial difference of our method is that we do not impose any shape or smoothness assumptions on the \emph{boundary} of the object. This permits the detection of nonsmooth, irregular or disconnected objects in noisy images, under very mild assumptions on the object's interior. This is especially suitable, for example, if one has to detect a highly irregular non-convex object in a noisy image. This is usually the case, for example, in the aforementioned fields of automated urban analysis, cancer detection and detection of cracks in materials. Although our detection procedure works for regular images as well, it is precisely the class of irregular images with unknown shape where our method can be very advantageous.

\subsection{Outline}

The paper is organized as follows. Our statistical model is described in details in Section \ref{Section_Model}. Proper type of thresholding for noisy images is crucial in our method and would allow us to apply percolation theory to our learning task. Both thresholding and percolation are described in Section \ref{Section_Thresholding}. An algorithm for object detection is presented in Section \ref{Section2}. Theorem \ref{Theorem1} established the strong consistency and scalability of the method in the setup with minimal assumptions on both the noise and the object of interest. An example illustrating possible applications of our method is given in Section \ref{Section_Example}. Section \ref{Proofs_Triangular_Lattice} is devoted to the proof of the main theorem.


\section{Statistical model}\label{Section_Model}

Assume we observe a noisy digital image on a screen of $N \times N$ pixels. Object detection and image reconstruction for noisy images are two of the cornerstone problems in image analysis. In this paper, we develop an efficient scalable robust technique for quick detection of objects in noisy images. 


In the present paper we are interested in detection of objects that have a known colour. This colour has to be different from the colour of the background. Mathematically, this is equivalent to assuming that the true (non-noisy) images are black-and-white, where the object of interest is black and the background is white.

Without loss of generality, we are free to assume that all the pixels that belong to the meaningful object within the digitalized image have the value 1 attached to them. We can call this value a \emph{black colour}. Additionally, assume that the value 0 is attached to those and only those pixels that do not belong to the object in the non-noisy image. If the number 0 is attached to the pixel, we call this pixel \emph{white}.


It is also assumed that on each pixel we have random noise that has the \emph{unknown} distribution function  $F$; the noise at each pixel is completely independent from noises on other pixels. It is important that we consider the case of a fully nonparametric noise of \emph{unknown} level and having an \emph{unknown} distribution.


More formally, we have an $N \times N$ array of observations, i.e. we observe $N^2$ real numbers ${\{Y_{ij}\}}_{i,j=1}^N$. Denote the true value on the pixel $(i, j)$, $1 \leq i, j \leq N$, by $Im_{ij}$, and the corresponding noise by $\sigma \varepsilon_{ij}$. According to the above,

\begin{equation}\label{1}
Y_{ij} = Im_{ij} + \sigma \, \varepsilon_{ij}\,,
\end{equation}

\noindent where $1 \leq i, j \leq N$, and $\{ \varepsilon_{ij}\}$, $1 \leq i, j \leq N$ are i.i.d., and

\begin{equation}\label{3}
Im_{ij} = \left\{
           \begin{array}{ll}
             1, & \hbox{if $(i,j)$ belongs to the object;} \\
             0, & \hbox{if $(i,j)$ does not belong to the object.}
           \end{array}
         \right.
\end{equation}

\noindent To stress the dependence on the noise level $\sigma$, we write our assumption on the noise in the following way:

\begin{equation}\label{2}
\varepsilon_{ij} \sim F, \quad \mathbb{E}\, \varepsilon_{ij} = 0,  \quad Var\, \varepsilon_{ij} = 1\,.
\end{equation}


\noindent Throughout this paper we will additionally assume that the following non-degeneracy assumption holds. \medskip

\begin{equation}\label{22}
\langle  \textbf{A} \rangle\quad\quad\quad F (t) \equiv C \quad\text{for all}\quad t \in (a, b) \quad \Rightarrow \quad b-a < 1.
\end{equation}
\medskip

\noindent The \emph{null hypothesis} is $H_0:\,Im_{i, j} = 0$ for all $i, j$. The alternative hypothesis is $H_1:\, Im_{ij} \neq 0$ for some $i, j$.


Now we can proceed to preliminary quantitative estimates. If a pixel $(i, j)$ is white in the original image, let us denote the corresponding probability distribution of $Y_{ij}$ by $P_0$. For a black pixel $(i, j)$ we denote the corresponding distribution of $Y_{ij}$ by $P_1$. We are free to omit dependency of $P_0$ and $P_1$ on $i$ and $j$ in our notation, since all the noises are independent and identically distributed.

The following simple observation will be used to link community detection in triangular networks and percolation theory.

\begin{proposition}\label{Proposition_1}
If $\langle A \rangle$ holds and the distribution of the noise distribution is symmetric, then

\begin{eqnarray}
  P_0 (\,Y_{ij} \geq 1/2\,) &<& 1/2\,,\label{23} \\
  1/2 &<& P_1 (\,Y_{ij} \geq 1/2\,)\,.\label{24}
\end{eqnarray}

\end{proposition}

\begin{proof}(Proposition \ref{Proposition_1})
Since the noise is symmetric, assumption $\langle A \rangle$ yields

\begin{eqnarray*}
  P_1 (\,Y_{ij} \geq 1/2\,) &=& P(\,\varepsilon + 1 > 1/2\,) \\
    &=& P(\,\varepsilon > - 1/2\,) \\
    &=& P(\,\varepsilon < 1/2\,)   \\
    &>& P(\,\varepsilon < 0\,) = 1/2\,.
\end{eqnarray*}

\noindent For the other part, we have in view of the previous calculation

\begin{eqnarray*}
  P_0 (\,Y_{ij} \geq 1/2\,) &=& P(\,\varepsilon  \geq 1/2\,) \\
    &=& 1 - P(\,\varepsilon < 1/2\,) \\
    &<& 1/2\,.
\end{eqnarray*}

\noindent This completes the proof.
\end{proof}

\section{Thresholding and percolation on triangular lattices}\label{Section_Thresholding}

As was shown in \citep{Langovoy_Wittich_Square}, \citep{langovoy_wittich_robust}, \citep{Arias-Castro_Grimmett}, percolation-based detection methods are applicable to more general types of networks than lattices or regular graphs. In this paper, in order to obtain strong stability against a very wide nonparametric class of noise distributions, we decided to stick to graphs with critical probability $0.5$, of which the triangular lattice is the most natural example.

\subsection{Thresholding}

Now we are ready to describe one of the main ingredients of our method: the \emph{thresholding}. The idea of the thresholding is as follows: in the noisy grayscale image ${\{Y_{ij}\}}_{i,j=1}^N$, we pick some pixels that look as if their real colour was black. Then we colour all those pixels black, irrespectively of the exact value that was observed on them. We take into account the intensity observed at those pixels only once, in the beginning of our procedures. The idea is to think that some pixel "seems to have a black colour" when it is not very likely to obtain the observed grey value when adding a "reasonable" noise to a white pixel.

We colour white all the pixels that weren't coloured black at the previous step. At the end of this procedure, we would have a transformed vector of 0's and 1's, call it $\{\overline{Y}_{i,j}\}_{i,j=1}^{N}$. We will be able to analyse this transformed picture by using certain results from the mathematical theory of percolation.


Let us fix, for each $N$, a real number $\alpha_0(N) > 0$, $\alpha_0(N) \leq 1$, such that there exists $\theta (N) \in \mathbb{R}$ satisfying the following condition:

\begin{equation}\label{6}
P_0 (\,Y_{ij} \geq \theta (N)\, ) \,\leq\, \alpha_0(N)\,.
\end{equation}

In this paper we will always pick $\alpha_0(N)\,\equiv\,\alpha_0$ for all $N \in \mathbb{N}$, for some constant $\alpha_0 > 0$.


As a first step, we transform the observed noisy image $\{Y_{i,j}\}_{i,j=1}^{N}$ in the following way: for all $1 \leq i, j \leq N$,\par\smallskip

1. $\quad$ If $Y_{ij} \geq \theta (N)$, set $\overline{Y}_{ij} := 1$ (i.e., in the transformed picture the corresponding pixel is coloured black).\smallskip

2. $\quad$ If $Y_{ij} < \theta (N)$, set $\overline{Y}_{ij} := 0$ (i.e., in the transformed picture the corresponding pixel is coloured white).\smallskip

\begin{definition}\label{Definition1}
The above transformation is called \emph{thresholding at the level} $\theta (N)$. The resulting array $\{\overline{Y}_{i,j}\}_{i,j=1}^{N}$ of $N^2$ values (0's and 1's) is called a \emph{thresholded picture}.
\end{definition}
\smallskip

\subsection{Percolation}

One can think of pixels from $\{\overline{Y}_{i,j}\}_{i,j=1}^{N}$ as of vertices of a planar graph. Let us colour these $N^2$ vertices with the same colours as the corresponding pixels. We obtain a graph $G_N$ with $N^2$ black or white vertices and (so far) no edges.

We add edges to $G_N$ in the following way. If any two \emph{black} vertices are "neighbours" (in a way to be specified below), we connect these two vertices with a \emph{black} edge. If any two \emph{white} vertices are neighbours, we connect them with a \emph{white} edge. We will not add any edges between non-neighbouring points, and we will not connect vertices of different colours to each other.

It is crucial how one defines \emph{neigbourhoods} for vertices of $G_N$: different definitions can lead to testing procedures with very different properties. The first and a very natural way is to view $G_N$ as an $N \times N$ square subset of the $\mathbb{Z}^2$ lattice. The method works in this case, see \citep{langovoy_report_2009-035} and \citep{Langovoy_Wittich_Square}. However, it turns out that the method becomes especially robust when we view our black and white pixelized picture as a collection of black and white clusters on an $N \times N$ subset of the \emph{triangular} lattice $\mathbb{T}^2$ (obtained from $\mathbb{Z}^2$ lattice by adding diagonals to every square on the lattice). In the present paper, we will work exclusively with triangular lattices.

We perform $\theta (N)-$thresholding of the noisy image $\{Y_{i,j}\}_{i,j=1}^{N}$ using with a very special value of $\theta (N)$. Our goal is to choose $\theta (N)$ (and corresponding $\alpha_0(N)$, see (\ref{6})) such that:

\begin{eqnarray}
  P_0 (\,Y_{ij} \geq \theta (N)\,) &<& p_{c}^{site}\,,\label{8} \\
  p_{c}^{site} &<& P_1 (\,Y_{ij} \geq \theta (N)\, )\,,\label{9}
\end{eqnarray}

\noindent where $p_{c}^{site}$ is the critical probability for site percolation on $\mathbb{T}^2$ (see \citep{Grimmett}, \citep{Kesten}).


Since $G_N$ is random, we actually observe the so-called \emph{site percolation} on black vertices within the subset of $\mathbb{T}^2$. From this point, we can use results from percolation theory to predict formation of black and white clusters on $G_N$, as well as to estimate the number of clusters and their sizes and shapes. Relations (\ref{8}) and (\ref{9}) are crucial here.

To explain this more formally, let us split the set of vertices $V_N$ of the graph $G_N$ into to groups: $V_N = V_{N}^{im} \cup V_{N}^{out}$, where $V_{N}^{im} \cap V_{N}^{out} = \emptyset$, and $V_{N}^{im}$ consists of those and only those vertices that correspond to pixels belonging to the original object, while $V_{N}^{out}$ is left for the pixels from the background. Denote $G_{N}^{im}$ the subgraph of $G_N$ with vertex set $V_{N}^{im}$, and denote $G_{N}^{out}$ the subgraph of $G_N$ with vertex set $V_{N}^{out}$.

If (\ref{8}) and (\ref{9}) are satisfied, we will observe a so-called \emph{supercritical percolation} of black clusters on $G_{N}^{im}$, and a \emph{subcritical} percolation of black clusters on  $G_{N}^{out}$. Without going into much details on percolation theory (the necessary introduction can be found in \citep{Grimmett} or \citep{Kesten}), we mention that there will be a high probability of forming relatively large black clusters on $G_{N}^{im}$, but there will be only little and scarce black clusters on $G_{N}^{out}$. The difference between the two regions will be striking, and this is the main component in our image analysis method.

In this paper, mathematical percolation theory will be used to derive quantitative results on behaviour of clusters for both cases. We will apply those results to build efficient randomized algorithms that will be able to detect and estimate the object $\{Im_{i,j}\}_{i,j=1}^{N}$ using the difference in percolation phases on $G_{N}^{im}$ and $G_{N}^{out}$.

But when can the key inequalities (\ref{8}) and (\ref{9}) be simultaneously satisfied for an appropriate threshold $\theta$? The following important proposition shows that, under very mild conditions, our method is asymptotically consistent for any noise level.

\begin{proposition}\label{Proposition_2}
On the triangular lattice (\ref{8}) and (\ref{9}) are always satisfied for $\theta = 1/2$.
\end{proposition}

\begin{proof}(Proposition \ref{Proposition_2})
For the planar triangular lattice one has $p_{c}^{site} = 1/2$ (see \citep{Kesten}). The statement follows from Proposition \ref{Proposition_1}.
\end{proof}

Proposition \ref{Proposition_2} explains the main reason for working with the triangular lattice: for this lattice, the method is asymptotically consistent for any noise level, and the natural threshold $\theta (N) = 1/2$ is always appropriate. As we will see in the following section, this makes our testing procedure applicable in the case of unknown and nonsmooth nonparametric noise.

\section{Object detection}\label{Section2}

We either observe a blank white screen with accidental noise or there is an actual object in the blurred picture. In this section, we propose an algorithm to make a decision on which of the two possibilities is true. This algorithm is a statistical testing procedure. It is designed to solve the question of testing $H_0:\, I_{ij}=0\,\, \mbox{for all}\,\, 1 \leq i, j \leq N$ versus $H_1:\, I_{ij} = 1 \,\,\mbox{for some}\,\, i, j$.

Let us choose $\alpha (N) \in (0, 1)$ - the \emph{probability of false detection} of an object. More formally, $\alpha (N)$ is the maximal probability that the algorithm finishes its work with the decision that there was an object in the picture, while in fact there was just noise. In statistical terminology, $\alpha (N)$ is the probability of an error of the first kind. We allow $\alpha$ to depend on $N$; $\alpha(N)$ is connected with complexity (and expected working time) of our randomized algorithm.

Since in our method it is crucial to observe some kind of percolation in the picture (at least within the image), the image has to be "not too small" in order to be detectable by the algorithm: one can't observe anything percolation-alike on just a few pixels. We will use percolation theory to determine how "large" precisely the object has to be in order to be detectable. Some size assumption has to be present in any detection problem, though: for example, it is mathematically hopeless to detect a single point on a very large screen even in the case of a moderate noise. We proceed with the following weak assumption about the object's interior part: \par\smallskip



$\langle \textbf{B} \rangle \quad$ The object contains a black square with the side of size at least $\varphi_{im}(N)$ pixels, where


\begin{equation}\label{12}
\lim_{N\to\infty} \frac{\,\log \frac{1}{\,\alpha(N)\,}\,}{\,\varphi_{im}(N)\,}\, = 0\,.
\end{equation}


\begin{equation}\label{14}
 \lim_{N \to\infty} \frac{\,\varphi_{im} (N)\,}{\,\log N\,} = \infty\,.
\end{equation}
\medskip



\noindent Assumption $\langle B \rangle$ is a \emph{sufficient} condition for the algorithm to work. For example, it is possible to relax (\ref{14}) and to replace a square in $\langle B \rangle$ by a triangle-shaped figure. Although conditions (\ref{12}) and (\ref{14}) are of asymptotic character, most of the estimates used in our method are valid for finite $N$ as well. For obvious consistency reasons, $\varphi_{im}(N)\, \leq N\,$.


Now we are ready to formulate our \emph{Detection Algorithm} (see Algorithm 1). Fix the false detection rate $\alpha (N)$ before running the algorithm.

\begin{algorithm}[t]
   \caption{Detection}
   \label{Algorithm_Detection}
\begin{algorithmic}[1]
   \STATE \textsf{\small Step 0.} Find an optimal $\theta (N)$ (in our framework $\theta (N) := 1/2$).
   \STATE \textsf{\small Step 1.} Perform $\theta(N)-$thresholding of the noisy picture $\{Y_{i,j}\}_{i,j=1}^{N}$.
   \STATE \textsf{\small Step 2.} \REPEAT \STATE Run depth-first search \citep{Tarjan} on the graph $G_N$ of the $\theta(N)-$thresholded picture $\{{\overline{Y}}_{i,j}\}_{i,j=1}^{N}$
   \UNTIL{\{\{Black cluster of size $\varphi_{im} (N) $ is found\} or \{all black clusters are found\}\}}
   \STATE \textsf{\small Step 3.}
   \IF{black cluster of size $\varphi_{im} (N) $ was found}
   \STATE Output: an object was detected
   \ELSE
   \STATE Output: there is no object.
   \ENDIF
\end{algorithmic}
\end{algorithm}

%
%
%
%
%
%
%
%
%
%

\noindent At Step 2 our algorithm finds and stores not only sizes of black clusters, but also coordinates of pixels constituting each cluster. We remind that $\theta(N)$ is defined as in (\ref{6}), $G_N$ and $\{{\overline{Y}}_{i,j}\}_{i,j=1}^{N}$ were defined in Section \ref{Section_Thresholding}, and $\varphi_{im} (N)$ is any function satisfying (\ref{12}). The depth-first search algorithm is a standard procedure used for searching connected components on graphs. This procedure is a deterministic algorithm. The detailed description and rigorous complexity analysis can be found in \citep{Tarjan}, or in the classic book \citep{Aho_Hopcroft_Ullman}, Chapter 5.

Let us prove that Algorithm 1 works, and determine its complexity.

\begin{theorem}\label{Theorem1}
Suppose assumptions $\langle A \rangle$ and $\langle B \rangle$ are satisfied and the noise is symmetric. Then\smallskip

\begin{enumerate}
  \item Algorithm 1 finishes its work in $O(N^2)$ steps, i.e. is linear.\medskip

  \item If there was an object in the picture, Algorithm 1 detects it with probability at least $(1 - \exp(-C_1(\sigma) \varphi_{im} (N)))$.\medskip

  \item The probability of false detection doesn't exceed $\min\{\alpha (N), \exp(-C_2(\sigma) \varphi_{im} (N)) \}$ for all $N > N(\sigma)$.
\end{enumerate}

\noindent The constants $C_1 > 0$, $C_2 > 0$ and $N(\sigma)\in\mathbb{N}$ depend only on $\sigma$.

\end{theorem}

Theorem \ref{Theorem1} means that Algorithm 1 is of quickest possible order: it is \emph{linear} in the input size in the worst case. Theorem \ref{Theorem1} also implies that the algorithm has computational complexity $O(\varphi_{im}(N))$ if the starting point of the depth-first search was close enough to the object. It is difficult to think of an algorithm working quicker in this problem. Indeed, if the image is very small and located in an unknown place on the screen, or if there is no image at all, then any algorithm solving the detection problem will have to at least upload information about $O(N^2)$ pixels, i.e. under general assumptions of Theorem \ref{Theorem1}, any detection algorithm will have at least linear complexity.

Another important point is that Algorithm 1 is not only consistent, but that it has \emph{exponential} rate of accuracy, typically achievable only for  parametric or sufficiently smooth models.

\begin{remark}
\emph{It is also interesting to remark here that, although it is assumed that the object of interest contains a $\varphi_{im}(N) \times \varphi_{im}(N)$ black square, one cannot use a very natural idea of simply considering sums of values on all squares of size $\varphi_{im}(N) \times \varphi_{im}(N)$ in order to detect an object. Neither some sort of thresholding can be avoided, in general. Indeed, although this simple idea works very well for normal noise, it cannot be used in case of unknown and possibly irregular or heavy-tailed noise. For example, for heavy-tailed noise, detection based on non-thresholded sums of values over subsquares will lead to a high probability of false detection. Whereas the method of the present paper still works.}
\end{remark}

\section{Example}\label{Section_Example}

In this section, we outline an example illustrating possible applications of our method. We start with a real greyscale picture of a neuron (see Fig. \ref{Picture_Neuron}). This neuron is an irregular object with unknown shape, and our method can be very advantageous in situations like this.

\begin{figure}[h]
\begin{center}
\includegraphics[width=90mm]{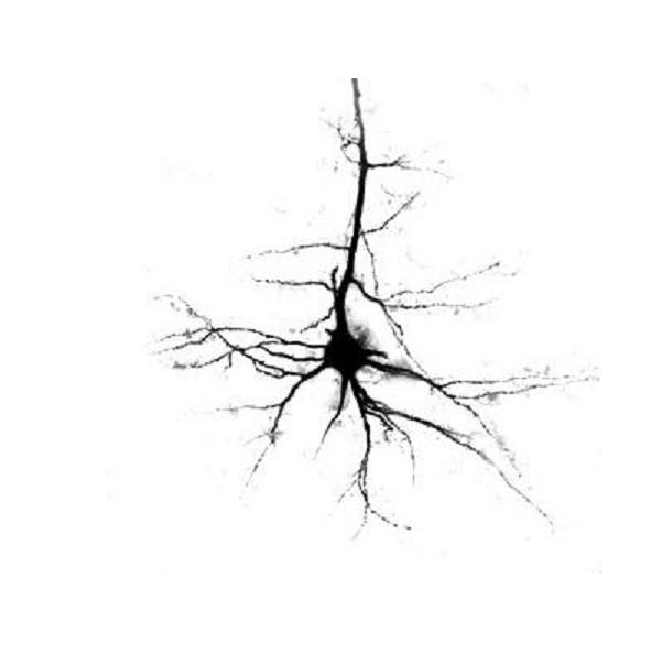}
\caption{A part of a real neuron.}
\label{Picture_Neuron}
\end{center}
\end{figure}

Basing on this real picture, we perform the following simulation study. We add Gaussian noise of level $\sigma = 1.8$ independently to each pixel in the image, and then we run Algorithm 1 on this noisy picture. A typical version of a noisy picture with this relatively strong noise can be seen on Fig. \ref{Picture_Neuron_Noise}. We run the algorithm on 1000 simulated pictures. Note that we used Gaussian noise for illustrative purposes only. We did not make any use neither of the fact that the noise is normal nor of our knowledge of the actual noise level.

As a result, the neuron was detected in 96.8\% of all cases. At the same time, the probability of false detection was shown to be below 5\%. Now we describe our experiment in more details.

\begin{figure}[h]
\begin{center}
\includegraphics[width=90mm]{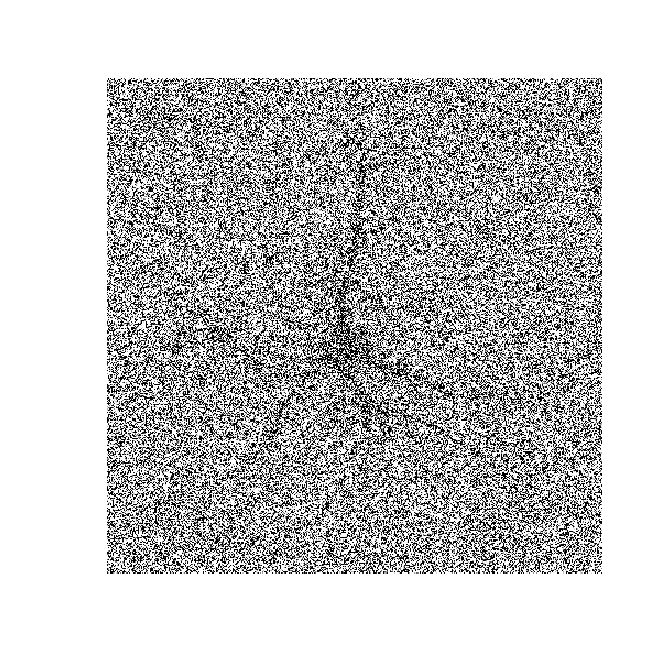}
\caption{A noisy picture.}
\label{Picture_Neuron_Noise}
\end{center}
\end{figure}

The starting picture (see Fig. \ref{Picture_Neuron}) was $450 \times 450$ pixels. White pixels have value 0 and black pixels have value 1. Some pixels were grey already in the original picture, but in practice this doesn't spoil the detection procedure.

We used as a threshold $\theta = 0.5$. The thresholded version of Fig. \ref{Picture_Neuron_Noise} is shown on Fig. \ref{Picture_Neuron_Threshold}. As follows from Theorem \ref{Theorem1}, our testing procedure is asymptotically consistent. We have chosen $\sigma = 1.8$ in our simulation study. In practice, Algorithm 1 can be consistently used for stronger noise levels for images of this size.

\begin{figure}[h]
\begin{center}
\includegraphics[width=90mm]{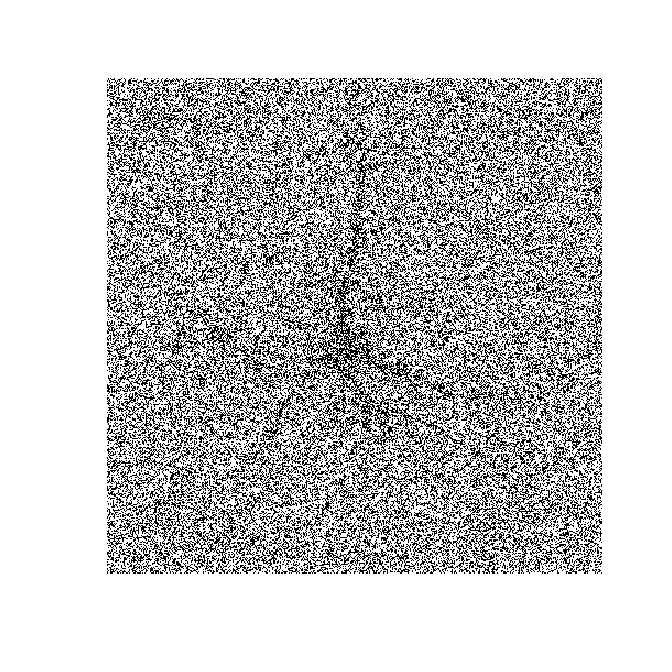}
\caption{A thresholded picture.}
\label{Picture_Neuron_Threshold}
\end{center}
\end{figure}

Suppose the null hypothesis is true, i.e. there is no signal in the original picture. By running Algorithm 1 on empty pictures of size $450 \times 450$ with simulated noise of level $\sigma = 1.8$ and $\theta = 0.5$, one can find that with probability more than 95\% there will be no black cluster of size 304 or more on the thresholded picture. Therefore, we considered as significant only those clusters that had more than 304 pixels. A different and much more efficient way of calculating $\varphi (N)$ for moderate sizes of $N$ is proposed in \citep{langovoy_wittich_report_R}.

For moderate sample sizes, the algorithm is applicable in many situations that are not covered by Theorem 1. The object, of course, doesn't have to contain a square of size $303 \times 303$ in order to be detectable. In particular, for noise level $\sigma = 1.8$, even objects containing a $40 \times 40$ square are consistently detected. The neuron on Fig. \ref{Picture_Neuron} passes this criterion, and Algorithm 1 detected the neuron 968 times out of 1000 runs.


We remark here that the algorithm is also very quick in practice. For example, its realization in Python typically requires less than 1 second to process a 4000 by 4000 image on a personal computer.

\section{Proofs}\label{Proofs_Triangular_Lattice}

Before proving the main result, we shall state first the following theorem about subcritical site percolation on the standard triangular lattice $\mathbb{T}^2$.

\begin{theorem}\label{Theorem3}
Consider site percolation with probability $p_0$ on $\mathbb{T}^2$. There exists a constant $\lambda_{site} = \lambda_{site}(p_0) > 0$ such that

\begin{equation}\label{17}
P_{p_0} (\,|C| \geq n\,)\,\leq\, e^{-n\,\lambda_{site}(p_0)}\,,\quad\mbox{for all}\quad n \geq N(p_0)\,.
\end{equation}

\noindent Here $C$ denotes the open cluster containing the origin.
\end{theorem}

\begin{proof} (Theorem \ref{Theorem3}): The triangular lattice satisfies conditions of the Theorem 5.1 in \citep{Kesten}, p.83. Therefore, the second part of that Theorem (see equations (5.12)-(5.14) and the conclusion following them) ensures that there exist constants $C_1 = C_1 (p_0) >0$, $C_2 = C_2 (p_0) >0$ such that

\begin{equation}\label{25}
P_{p_0} (\,|C| \geq n\,)\,\leq\, C_1(p_0)\,e^{-n\,C_2(p_0)}\,,\quad\mbox{for all}\quad n \geq 1\,.
\end{equation}

\noindent If $C_1 \leq 1$, then (\ref{17}) follows immediately. Otherwise, (\ref{17}) follows from (\ref{25}) for all $n \geq N(p_0)$ and any $\lambda_{site}(p_0):= C_3 = C_3(p) > 0$ such that $N(p_0)$ and $C_3$ satisfy the inequality

\begin{equation}\label{26}
N(p_0)\,(C_2 - C_3)\,\geq\,\log C_1\,.
\end{equation}
\end{proof}

We will also need to use the celebrated FKG inequality (see \citep{FKG}, or \citep{Grimmett}, Theorem 2.4, p.34; see also Grimmett's book for some explanation of the terminology).

\begin{theorem}\label{TheoremFKG}
If $A$ and $B$ are both increasing (or both decreasing) events on the same measurable pair $(\Omega , \mathcal{F})$, then

\[
P(A \cap B)\,\geq\, P(A)\,P(B)\,.
\]
\end{theorem}
\medskip

Define $F_N(n)$ as the event that there is an \emph{erroneously marked} black cluster of size greater or equal $n$, lying in the square of size $N \times N$ corresponding to the screen. (An erroneously marked black cluster is a black cluster on $G_N$ such that \emph{each} of the pixels in the cluster was wrongly coloured black after the $\theta-$thresholding).

Denote

\begin{equation}\label{15}
p_{out}(N)\,:=\,P(\,Y_{ij} \geq 1/2\, | \,Im_{ij} =0\, )\,,
\end{equation}

\noindent a probability of erroneously marking a white pixel outside of the image as black.

The next theorem is particularly useful when studying percolation on finite sublattices of the initial infinite lattice.

\begin{theorem}\label{Theorem2}
Suppose that $0 < p_{out}(N) < p_{c}^{site}$. There exists a constant $C_3 = C_3 (p_{out}(N)) > 0$ such that

\begin{equation}\label{16}
P_{p_{out}(N)} (F_N(n))\,\leq\,\exp (\,-n \,C_3 (p_{out}(N)))\,,\quad\mbox{for all}\quad n \geq \varphi_{im} (N)\,.
\end{equation}
\end{theorem}

\begin{proof} (Theorem \ref{Theorem2}):
Denote by $C(i,j)$ the largest cluster in the $N\times N$ screen (triangulated by diagonals of one orientation) containing the pixel with coordinates $(i, j)$, and by $C(0)$ the largest black cluster on the same $N \times N$ screen containing 0. It doesn't matter for this proof which point is denoted by 0. By Theorem \ref{Theorem3}, for all $i$, $j$: $1 \leq i, j \leq N$:

\begin{eqnarray}\label{19}
  P_{p_{out}(N)} (\,|C(0)| \geq n\,) & \leq & e^{-n\,\lambda_{site}(p_{out})}\,, \\
  P_{p_{out}(N)} (\,|C(i, j)| \geq n\,) & \leq & e^{-n\,\lambda_{site}(p_{out})}\,. \nonumber
\end{eqnarray}

\noindent Obviously, it only helped to inequalities (\ref{17}) and (\ref{19}) that we have limited our clusters to only a finite subset instead of the whole lattice $\mathbb{T}^2$. On a side note, there is no symmetry anymore between arbitrary points of the $N \times N$ finite subset of the triangular lattice; luckily, this doesn't affect the present proof.

Since $\{\,|C(0)| \geq n\,\}$ and $\{\,|C(i, j)| \geq n\,\}$ are increasing events (on the measurable pair corresponding to the standard random-graph model on $G_N$), we have that $\{\,|C(0)| < n\,\}$ and $\{\,|C(i, j)| < n\,\}$ are decreasing events for all $i$, $j$. By FKG inequality for decreasing events,

\begin{eqnarray*}
  P_{p_{out}(N)} ( \,|C(i, j)| < n\, \,\mbox{for all}\,\, i, j, 1 \leq i, j \leq N\,) & \geq &  \\
  \prod\prod_{1 \leq i, j \leq N} P_{p_{out}(N)} ( \,|C(i, j)| < n\,) & \geq & (\mbox{by (\ref{19})}) \\
   & \geq & {\bigr(\,1 - e^{-n\,\lambda_{site}(p_{out})}\,\bigr)^{N^2}}\,.
\end{eqnarray*}

\noindent It follows that

\begin{eqnarray*}
  P_{p_{out}(N)} (F_N(n)) & = & P_{p_{out}(N)} \bigr( \,\exists (i, j),\, 1 \leq i, j \leq N\,: \,|C(i, j)| \geq n\, \bigr) \\
    & \leq & 1 -  {\bigr(\,1 - e^{-n\,\lambda_{site}(p_{out})}\,\bigr)^{N^2}} \\
    &=& 1 - \sum_{k=0}^{N^2} {(-1)}^k \,C_{N^2}^k \, e^{-n\,\lambda_{site}(p_{out})\,k} \\
    &=& \sum_{k=1}^{N^2} {(-1)}^{k-1} \,C_{N^2}^k \, e^{-n\,\lambda_{site}(p_{out})\,k} \\
    &=& N^2 e^{-n\,\lambda_{site}(p_{out})}\,+\,o\bigr(N^2 e^{-n\,\lambda_{site}(p_{out})}\bigr)\,,
\end{eqnarray*}

\noindent because we assumed in (\ref{16}) that $n \geq \varphi_{im} (N)$, and $\varphi_{im} (N) \gg \log N$. Moreover, we see immediately that Theorem \ref{Theorem2} follows now with some $C_3$ such that $0 < C_3 (p_{out}(N)) < \lambda_{site}(p_{out}(N))$.
\end{proof}

Now we establish the following useful lemma. Let $G_N$ denote the $N \times N$ subset of $\mathbb{T}^2$, as defined in Section \ref{Section_Thresholding} of the present paper. denote its canonical matching graph by $G_N^*$. We remind that $\mathbb{T}^2$ is self-matching, and refer to \citep{Kesten}, Section 2.2 for the necessary definitions. Assuming that $n \leq N$, denote $A_n$ be the event that there is an open (i.e., black) path in the rectangle $[0, n] \times [0, n]$ joining some vertex on its left side to some vertex on its right side. Similarly, let $B_n$ denote the event that there exists a closed (i.e., white) path on $G_n^*$ joining a vertex on the top side of $G_n^*$ to a vertex on its bottom side.

\begin{remark}\label{Imbedding_Crossings}
When speaking about black or white crossings of rectangles, we are free to assume that $\mathbb{T}^2$ is embedded in the plane as a $\mathbb{Z}^2$ lattice with diagonals. See \citep{Kesten} for a discussion of connections between percolation and various planar embeddings of regular lattices.
\end{remark}

\begin{lemma}\label{Lemma_Symmetry_Crossings} Let $0 < p < 1$ be a real number. Consider standard site percolation with probability $p$ on the triangular lattice. Then
\par
\begin{enumerate}

\item
Either $A_n$ or $B_n$ occurs. Moreover, $A_n \cap B_n = \emptyset\,.$

\item

\begin{equation}\label{27}
P_p (A_n) + P_p(B_n) = 1\,.
\end{equation}

\item

\begin{equation}\label{28}
P_p (A_n) + P_{1-p}(A_n) = 1\,.
\end{equation}

\end{enumerate}
\end{lemma}

\begin{proof} (Lemma \ref{Lemma_Symmetry_Crossings}). Statement 1 of the Lemma directly follows from Proposition 2.2 from \citep{Kesten} (see also pp.398 - 402 of that book: there a rigorous proof of this proposition is presented, including necessary topological considerations). Statement 2 is an immediate consequence of Statement 1 and definitions of percolation measures on $G_n$ and $G_n^*$.

To complete the proof, note that $G_N$ and $G_N^*$ are isomorphic, by Example (iii), pp. 19-20 of \citep{Kesten}. Since by definition a vertex of $G_n^*$ is black with probability $1-p$, we have that

\begin{equation}\label{29}
P_p (B_n) = P_{1-p}(A_n)\,.
\end{equation}

This proves (\ref{28}).
\end{proof}

First we prove the following theorem:

\begin{theorem}\label{Theorem4}
Consider site percolation on $\mathbb{T}^2$ lattice with percolation probability $p > p_c^{site} = 1/2$. Let $A_n$ be the event that there is an open path in the rectangle $[0, n] \times [0, n]$ joining some vertex on its left side to some vertex on its right side. Let $M_n$ be the maximal number of vertex-disjoint open left-right crossings of the rectangle $[0, n] \times [0, n]$. Then there exist constants $C_4 = C_4 (p) > 0$, $C_5 = C_5 (p) > 0$, $C_6 = C_6 (p) > 0$ such that

\begin{equation}\label{20}
P_p (A_n) \,\geq\, 1 - (n+1)\,e^{-C_4\,n}\,,
\end{equation}

\begin{equation}\label{21}
P_p (\,M_n \leq C_5\,n\,) \,\leq\, e^{-C_6\,n}\,,
\end{equation}

\noindent and both inequalities holds for all $n \geq N_1 (p)$.
\end{theorem}

\begin{proof} (Theorem \ref{Theorem4}): Let $LR_k(n)$, $0 \leq k \leq n$, be the event that the point $(0, k)$ of $G_n$ is connected by a \emph{white} (in other words, closed) path (that lies in the interior of $G_n$) to some vertex on the right border of $G_n$. Denote by $LR(n)$ the event that there exists a closed left-right crossing of $G_n$. Let $C((0, k))$ denotes the white cluster containing the point $(0, k)$, where we make a convention that this cluster is considered on the whole lattice $\mathbb{T}^2$. Then obviously

\begin{equation}\label{30}
    LR_k(n) \,\subseteq\, \{\,\omega:\, |C((0, k))| \geq n \,\}
\end{equation}

\noindent and

\begin{equation}\label{31}
    LR(n) \,\subseteq\, \bigcup_{k=0}^{n}\, LR_k(n) \,.
\end{equation}

\noindent Now (\ref{31}) gives us

\begin{equation}\label{32}
P_{1-p}(LR(n)) \,\leq\, \sum_{k=0}^{n} P_{1-p}(LR_k(n)) \,\leq\, (n+1)\,\max_{0 \leq k \leq 1} P_{1-p}(LR_k(n))\,.
\end{equation}

\noindent Since $1 - p < p_c^{site}$, we get from (\ref{30}) and Theorem \ref{Theorem3} that for all $k$

\begin{equation}\label{33}
P_{1-p}(LR_k(n)) \,\leq\, P_{1-p}(\{\, |C((0, k))| \geq n \,\}) \,\leq\, e^{-C_4\,n}\,.
\end{equation}

\noindent Combining (\ref{32}) and (\ref{33}) yields

\begin{equation}\label{34}
P_{1-p}(A_n) \,=\, P_{1-p}(LR(n)) \,\leq\, (n+1)\,e^{-C_4\,n}\,.
\end{equation}

\noindent Altogether, (\ref{28}) and (\ref{34}) imply (\ref{20}). This proves the first half of Theorem \ref{Theorem4}.

As about the second part of the proof, (\ref{21}) is deduced from (\ref{20}) with the help of Theorem 2.45 of \citep{Grimmett}. The derivation itself is presented at pp. 49-50 of \citep{Grimmett}; the only difference is that in our case one has to change "edges" by "vertices" in the proof from the book. Everything else works the same, since Theorem 2.45 is valid for all Bernoulli product measures on regular lattices; in particular, Theorem 2.45 applies for site percolation as well. This completes the proof of Theorem \ref{Theorem4}.
\end{proof}

\begin{proof} (Theorem \ref{Theorem1}): $\quad$ I.$\quad$ First we prove the complexity result.

The $\theta(N)-$thresholding gives us $\{{\overline{Y}}_{i,j}\}_{i,j=1}^{N}$ and $\overline{G}_N$ in $O(N^2)$ operations. This finishes the analysis of Step 1.

As for Step 2, it is known (see, for example, \citep{Aho_Hopcroft_Ullman}, Chapter 5, or \citep{Tarjan}) that the standard depth-first search finishes its work also in $O(N^2)$ steps. It takes not more than $O(N^2)$ operations to save positions of all pixels in all clusters to memory , since one has no more than $N^2$ positions and clusters. This completes analysis of Step 2 and shows that Algorithm 1 is linear in the size of the input data.\par\smallskip

II. Now we prove the bound on the probability of false detection. Denote

\begin{equation}\label{15}
p_{out}(N)\,:=\,P(\,Y_{ij} \geq 1/2\, | \,Im_{ij} =0\, )\,,
\end{equation}

\noindent a probability of erroneously marking a white pixel outside of the image as black. Under assumptions of Theorem \ref{Theorem1}, $p_{out}(N) < p_{c}^{site}$. The exponential bound on the probability of false detection follows trivially from Theorem \ref{Theorem2}.\par\smallskip

III. It remains to prove the lower bound on the probability of true detection. Suppose that we have an object in the picture that satisfies assumptions of Theorem \ref{Theorem1}. Consider any $\varphi_{im} (N) \times \varphi_{im} (N)$ square in this image. After $\theta-$thresholding of the picture by Algorithm 1, we observe on the selected square a site percolation with probability

\[
p_{im}(N) \,:=\, P(\,Y_{ij} \geq 1/2\, | \,Im_{ij} =1\, ) \,>\, p_{c}^{site}\,.
\]

\noindent Then, by (\ref{20}) of Theorem \ref{Theorem4}, there exists $C_4 = C_4 (p_{im}(N))$ such that there will be \emph{at least one} cluster of size not less than $\varphi_{im} (N)$ (for example, one could take any of the existing left-right crossings as a part of such cluster), provided that $N$ is bigger than certain $N_1(p_{im}(N))$; and all that happens with probability at least

\[
1 - n\,e^{-C_4\,n} \,>\, 1 - e^{-C_3\,n}\,,
\]

\noindent for some $C_3$: $0 < C_3 < C_4$. Note that one can always weaken the constant $C_3$ above in such a way that the estimate above starts to hold for all $n \geq 1$. Theorem \ref{Theorem1} is proved.

\end{proof}

\smallskip
\noindent {\bf Acknowledgments.} The authors would like to thank Remco van der Hofstad, Artem Sapozhnikov and Shota Gugushvili for helpful discussions. \\

\bibliographystyle{plainnat}
\bibliography{Randomized_Algorithms_and_Percolation}


\end{document}